\documentclass[12pt]{article}
\usepackage{amsmath,amsfonts,amssymb,amsthm}
\def\Z{\mathbb{Z}}

\def\Fix{\rm{FIX}}

\newtheorem{thm}{Theorem}

\newtheorem{lem}[thm]{Lemma}
\newtheorem{lemma}[thm]{Lemma}
\newtheorem{cor}[thm]{Corollary}

\newtheorem{prop}[thm]{Proposition}

\newtheorem{defn}[thm]{Definition}
\newtheorem{rmk}[thm]{Remark}

\numberwithin{equation}{section}

\usepackage{color}

\def\sii{if and only if }
\def\ksc{{$k$-circular succession}}
\def\kscs{$k$-circular successions}
\def\ksl{$k$-linear succession}
\def\ksls{$k$-linear successions}
\def\ksl2{$k$-linear succession of second king }
\def\ksl2s{$k$-linear successions of second kind }
\def\phib{\Phi }
\def\si{\pi }

\def\siep'{\si'=(\varepsilon',\sigma')}
\def\sig'{sign_{\si'}}
\def\sig{sign_{\si }}
\def\sign{\textrm{sgn}}
\def\phi1b{\varphi^b }
\def\phid{\delta}
\title{Derangements and Euler's difference table for $C_\ell\wr S_n$}
\author{\small\sc{Hilarion L. M. Faliharimalala$^{1}$ and Jiang Zeng$^{2}$} \\[-0.8ex]
\\[-0.8ex]
\small $^1$D\'epartement de Math\'ematiques et Informatique\\ [-0.8ex]
 \small Universit\'e d'Antananarivo, 101 Antananarivo, Madagascar\\[-0.8ex]
 \small\texttt{hilarion@ist-tana.mg}\\[-0.8ex]
\\[-0.8ex]
 \small $^2$Universit\'e de Lyon, Universit\'e Lyon 1,\\[-0.8ex]
\small Institut Camille Jordan, UMR 5208 du CNRS, \\[-0.8ex]
\small F-69622, Villeurbanne Cedex, France\\[-0.8ex]
\small\texttt{zeng@math.univ-lyon1.fr}}
\date{}
\begin{document}
\maketitle

\begin{abstract}
Euler's difference table associated to the sequence $\{n!\}$ leads
 naturally to the counting formula for the derangements.
 In this paper we study
Euler's difference table associated to the sequence
$\{\ell^n n!\}$ and  the generalized derangement problem.
For the coefficients appearing in the later table
we will give
the combinatorial interpretations in terms of two kinds of $k$-successions of
the group $C_\ell\wr S_n$.
In particular  for $\ell=1$ we recover
the known results for
the symmetric groups  while for
$\ell=2$ we obtain  the corresponding results for the
 hyperoctahedral groups.
\end{abstract}
\section{Introduction}
The \emph{probl\`eme de rencontres}  in classical combinatorics consists in
 counting permutations without fixed points (see \cite[p. 9--12]{Co}). 
 On the other hand one finds in the works of Euler (see \cite{DR}) the following table 
of differences:
$$
g_{n}^n=n!\quad \text{and}\quad 
g_{n}^m=g_{n}^{m+1}-g_{n-1}^m \quad (0\leq m\leq n-1).
$$
Clearly this table  leads
 naturally to an explicit formula for $g_{n}^0$, which corresponds to 
 the number of derangements of $[n]=\{1,\ldots, n\}$.
As $n!$ is the cardinality of the symmetric group of $[n]$,
 Euler's
difference table   can be considered to be an array associated to the symmetric group.

In the last two decades much effort has been made to extend
various enumerative results on symmetric groups
 to other Coxeter groups,   the wreath product
of a cyclic group with a symmetric group, and  more generally to complex reflection groups.
The reader is referred to
 \cite{ar01,ABR,CG,Ch, FH,HLR,BG,BR,BB} and the references cited there for
  the recent works in this
 direction.
 
In this paper we shall consider
 the \emph{probl\`eme de rencontres} in the group
$C_\ell\wr S_n$ via Euler's difference table. 
For a fixed integer $\ell\geq 1$,  we define Euler's difference table
for $C_\ell\wr S_n$to be 
 the array
  $(g_{\ell,n}^m)_{n,\,m\geq 0}$ defined by
\begin{align}
\left\{%
\begin{array}{ll}
    g_{\ell,n}^n=\ell^n\, n!& \hbox{$(m=n)$;} \\
  g_{\ell,n}^m= g_{\ell,n}^{m+1}-g_{\ell,n-1}^m &
  \hbox{$(0\leq m\leq n-1)$.}
\end{array}%
\right.\label{eulerB}
\end{align}
The first values of these numbers for
$\ell=1$ and $\ell=2$  are given in Table 1.

 \begin{table}[h]
$$\vcenter{
\hbox{$\begin{array}{c|cccccc}
\hbox{$n$}\backslash\hbox{$m$}&0&1&2&3&4&5\\
\hline
0&1& & & & &\\
1&0&1!& & &&\\
2&1&1&2!& &&\\
3&2& 3& 4 &3!&&\\
4&9&11&14&18&4!&\\
5&44&53&64&78&96&5!\\
\end{array}$}
\smallskip
\hbox{\hskip 3cm $(g_{1,n}^m)$}}
\qquad
\vcenter{
\hbox{$\begin{array}{c|cccccc}
\hbox{$n$}\backslash\hbox{$m$}&0&1&2&3&4&5\\
\hline
0  &1&\\
1    & 1&2^1 \,1!&\\
2    &5&6&2^2 \,2!\\
3   &29&34&40&2^3\,3!&\\
4    &233&262&296&336&2^4\,4!&\\
5&2329&2562&2824&3120&3456&2^5\,5!\\
\end{array}$}
\smallskip
\hbox{\hskip 3cm $(g_{2,n}^m)$}
}$$
\caption{Values of $g_{\ell,n}^m$ for $0\leq m\leq n\leq 5$ and $\ell=1$ or 2.\label{t:g}}
\end{table}

The $\ell=1$ case of (1.1) corresponds to Euler's difference
table, where
$g_{1,n}^n$ is the cardinality of $S_n$ and $g_{1,n}^0$ is the
number of \emph{derangements}, i.e., the fixed point free
permutations in $S_n$. The combinatorial interpretation for
 the  general coefficients $g_{1,n}^m$ was first studied by
Dumont and Randrianarivony~\cite{DR} and then by Clarke et al
\cite{CHZ}. More recently Rakotondrajao~\cite{Ra1,Ra2} has given
further combinatorial interpretations of these coefficients in
terms of $k$-\emph{successions} in symmetric groups.

As $g_{2,n}^n=2^nn!$ is the cardinality of the hyperoctahedral group $B_n$,
Chow~\cite{Ch} has given a  similar interpretation for $g_{2,n}^0$
in terms of derangements  in the hyperoctahedral groups.

For positive integers $\ell$ and $n$
the group of \emph{colored permutations}
of $n$ digits with $\ell$ colors is the wreath product
$G_{\ell,n}=C_\ell\wr S_n=C_\ell^n\rtimes S_n$,
  where
$C_\ell$ is the
$\ell$-cyclic group generated by $\zeta=e^{2i\pi/\ell}$
 and $S_n$  is the symmetric group of the set $[n]$.
By definition,
the multiplication in $G_{\ell,n}$, consisting of  pairs $(\epsilon,\sigma)\in C_\ell^n\times S_n$,
  is given by the following rule:
  for all $\si=(\epsilon,\sigma)$ and $\si'=(\epsilon',\sigma')$ in $G_{\ell,n}$,
 $$(\epsilon,\sigma)\cdot(\epsilon',\sigma')=
 ((\epsilon_1\epsilon'_{\sigma^{-1}(1)},
 \epsilon_2\epsilon'_{\sigma^{-1}(2)},\ldots,
  \epsilon_n\epsilon'_{\sigma^{-1}(n)}), \,\sigma\circ\sigma').
  $$
 One can identify $G_{\ell,n}$ with a
  permutation group of the colored set:
$$
\Sigma_{\ell,n}:=C_{\ell}\times [n]=\{\zeta^{j}i \, | \,  i\in[n], 0 \leq j\leq \ell-1\}
$$
via the morphism
$(\epsilon, \sigma)\longmapsto \pi$ such that   for any $i\in [n]$ and $0\leq j\leq \ell-1$,
$$
\pi(i)=\epsilon_{\sigma(i)}\sigma(i)\qquad\textrm{and}\qquad  \si(\zeta^j\, i)=\zeta^j\si(i).
$$
Clearly the cardinality of $G_{\ell,n}$ equals $\ell^n n!$.

 We can write a signed permutation $\pi\in G_{\ell,n}$ in two-line notation.
For example, if $\pi=(\epsilon, \sigma)\in G_{4,11}$, where
$\epsilon=(\zeta^2,1,1,\zeta,\zeta^2,\zeta, \zeta,\zeta,1,\zeta,\zeta^3)$ and
$$
\sigma=3\quad 5 \quad 1\quad  9\quad  6\quad  2\quad   7\quad  4\quad  11\quad  8\quad  10,
$$
we write
\begin{align*}
\pi=\left(%
\begin{array}{ccccccccccc}
  1 & 2 & 3 & 4 & 5 & 6 & 7 & 8 & 9 & 10 & 11 \\
  3 & \zeta^25 & \zeta^21 & 9 & \zeta 6 & 2 & \zeta7 & \zeta4 & \zeta^311 & \zeta 8 & \zeta 10 \\
\end{array}%
\right).
\end{align*}
For small $j$,  it is convenient to  write $j$ bars over $i$ instead of $\zeta^j i$.
Thus, the above permutation can be written  in one-line form as
$\pi=3 \;\; \bar{\bar 5} \;\; \bar{\bar{1}} \;\; 9 \;\; \bar 6 \;\; 2 \;\; \bar7 \;\;
 \bar 4 \;\; \overline{\overline{\overline{11}}} \;\; \bar8 \;\; \overline{10}$,
or in cyclic notation as
\begin{align*}
\pi=(\overline{\overline{1}}, \;3)\;(2,\;
\overline{\overline{5}},\; \overline{6})\; (\overline{4},\;9,\;
\overline{\overline{\overline{11}}},\;\overline{10},\;\overline{8})\;(\overline{7}).
\end{align*}
Note that when using cyclic notation to determine the image of a number, one ignores the sign on that number and then considers only the sign on the next number in the cycle. Thus,
in this example, we ignore the sign $\zeta^2$ on the 5 and note that then 5 maps to $\zeta6$ since the sign on 6 is $\zeta$.
 Furthermore, throughout this paper we shall use the following
 conventions:
\begin{itemize}
\item[i)] If $\si=(\epsilon,\sigma)\in G_{\ell,n}$,
let  $|\pi|=\sigma$  and $\sign_{\si}(i)=\epsilon_{i}$ for $i\in
[n]$. For example, if $\si= \bar{\bar{4}}\,\bar{3} \,1 \,\bar{2}$
 then $\epsilon=(1,\zeta,\zeta,\zeta^2)$ and $\sign_{\si}(4)=\zeta^2$.
\item[ii)] For  $i \in [n]$ and $ j
\in \{0,1,\ldots,\ell-1\}$ define
${\zeta^j}i+k={\zeta^j}(i+k)$ for  $ 0\leq k\leq n-i$, and
 $ {\zeta^j}i-k={\zeta^j}(i-k)$ for $0\leq k\leq i$.
For example,  we have $\bar{\bar{4}}+1=\bar{\bar{5}}$ in
$G_{4,11}$.
\item[iii)] We use the following total order on $\Sigma_{\ell,n}$:
for $i,j\in [\ell]$ and $a,b\in [n]$,
$$
\zeta^ia < \zeta^jb \Longleftrightarrow [i>j]\quad  \text{or}\quad
[i=j \text{ and } a<b].
$$
\end{itemize}

It is not hard to see that the coefficient $g_{\ell,n}^m$ is
divisible by $\ell^mm!$. This prompted us to introduce
 $d_{\ell,n}^m=g_{\ell,n}^m/\ell^mm!$. We derive then  from (1.1)
the following
allied array  $(d_{\ell,n}^m)_{n,\,m\geq 0}$:
\begin{align}
\left\{%
\begin{array}{ll}
    d_{\ell,n}^n=1& \hbox{$(m=n)$;} \\
  d_{\ell,n}^m= \ell(m+1)\, d_{\ell,n}^{m+1}-d_{\ell,n-1}^m&
  \hbox{$(0\leq m\leq n-1)$.}
\end{array}%
\right.\label{deulerB}
\end{align}
The first terms of these coefficients for
$\ell=1, 2$ are given in Table 2.

 \begin{table}[h]
$$
\vcenter{\hbox{
$\begin{array}{c|cccccccc}
\hbox{$n$}\backslash\hbox{$m$}&0&1&2&3&4&5\\
\hline
0  &1&\\
1  &0&1&\\
2  &1&1&1&\\
3  &2&3&2&1&\\
4  &9&11&7&3&1&\\
5  &44&53&32&13&4&1&
\end{array}$}
\smallskip
\hbox{\hskip 3cm $(d_{1,n}^m)$}
}
\qquad
\vcenter{
\hbox{$\begin{array}{c|cccccccccc}
\hbox{$n$}\backslash\hbox{$m$}&0&1&2&3&4&5\\
\hline
0& 1& & & & &\\
1& 1&1& & &&\\
2& 5&3&1& &&\\
3& 29&17&5&1&&&\\
4& 233&131&37&7&1&&&\\
5& 2329&1281&353&65&9&1&\\
\end{array}$}
\smallskip
\hbox{\hskip 3cm $(d_{2,n}^m)$}
}
$$
\caption{Values of $d_{\ell,n}^m$ for $0\leq m\leq n\leq 5$ and $\ell=1$ or 2.\label{t:g}}
\end{table}

One can find the $\ell=1$ case of \eqref{deulerB} and the table
$(d_{1,n}^m)$  in Riordan's
book~\cite[p. 188]{Ri}.  Recently Rakotondrajao~\cite{Ra2} has given a
combinatorial interpretation for the coefficients $d_{1,n}^m$ in
the symmetric group $S_n$.

The aim of this paper is
to study the coefficients $g_{\ell,n}^m$ and
 $d_{\ell,n}^m$ in the colored group $G_{\ell,n}$, i.e.,
 the wreath product of a cyclic group and a
symmetric group. This paper merges from the two papers \cite{CHZ} and
\cite{Ra2}. In the same vein as in \cite{CHZ}
we will give a $q$-version of \eqref{eulerB} in a forthcoming paper.

\section{Main results}

We first generalize the notion of $k$-succession introduced by Rakotondrajao~\cite{Ra2} in the symmetric group to $G_{\ell,n}$.
 \begin{defn}[$k$-circular succession]
Given  a permutation $\si\in G_{\ell,n}$ and a nonnegative integer $k $, the value $\pi(i)$ is
 a $k$-circular succession at position $i \in [n]$ if $\si(i)=i+k$.
  In particular a 0-circular succession is also called fixed point.
\end{defn}

\begin{rmk}
Some words are in order about the requirement $\si(i)=i+k$ in this definition.
 The ``wraparound`` is not allowed, i.e., $i+k$ is not to be interpreted $\mod n$,  also
$i+k$ needs to be uncolored, i.e., 
 $i+k\in [n]$,
 in order to count as a $k$-circular succession.
\end{rmk}

Denote by  $\mathcal{C}^k(\si)$ the set of  {\kscs} of $\si$ and let
 $c^k(\si)= \# \ \mathcal{C}^k(\si)$.
In particular $FIX(\si)$ denotes the set of fixed points of $\si$.
For example, for the permutation
$$ \si= \left ( \begin{array}{llllllllllllll}
1&2&3&4&5&6&7&8&9\\
\bar{1}&5 &\bar{\bar{9}}&\bar{6}&8 &\bar{7} &\bar{\bar{\bar{3}}}
&\bar{\bar{4}} &\bar{2}
\end{array}\right)\in G_{4,9},
$$
the values 5 and 8 are the two $3$-circular successions
at positions  2 and 5. Thus $C^3(\pi)=\{5, 8\}$.

The following is our main result on the combinatorial interpretation of the coefficients
$g_{\ell,n}^m$ in terms of $k$-circular successions.
\begin{thm}\label{succession majore} For any integer
 $k$ such that $0\leq k\leq m$, the entry $g_{\ell,n}^m$
equals the number of permutations in $G_{\ell,n}$ whose
 $k$-circular successions are included in $[m]$.
 In particular, by taking $k=0$ and $k=m$, respectively, either of the
 following holds.
 \begin{itemize}
\item[(i)] The entry  $g_{\ell,n}^m$ is the number of
   permutations in $G_{\ell,n}$ whose fixed points are included in $[m]$.
\item[(ii)] The entry  $g_{\ell,n}^m$ is the number of  permutations in $G_{\ell,n}$
without $m$-circular succession.
\end{itemize}
\end{thm}

For example, the  permutations in $G_{2,2}$ whose fixed points are
included in $[1]$ are:
$$
21, \quad 1\bar{2}, \quad \bar{2}1,\quad 2\bar{1},\quad
\bar{1}\bar{2},\quad \bar{2}\bar{1};
$$
while those  without $1$-circular succession  are:
$$
12, \quad \bar{1}2,\quad 1\bar{2},\quad \bar{1}\bar{2},\quad \bar{2}1,\quad \bar{2}\bar{1}.
$$

Note that Dumont and Randrianarivony~\cite{DR} proved the   $\ell=1$
case of (i), while Rakotondrajao~\cite{Ra2} proved the $\ell=1$
case of (ii).

Let $c_{\ell,n,m}^k$ be the number of colored permutations in $G_{\ell,n}$  with
 $m$ $k$-circular successions.

\begin{thm}\label{thm pol} Let $ n$, $k$ and $m$ be integers such that   $n\geq 1$, $k \geq 0$
 and  $m \geq 0$.
Then
\begin{align}
c_{\ell,n+1,m}^{k+1}&=c_{\ell,n+1,m}^k+c_{\ell,n,m}^k-c_{\ell,n,m-1}^k,
\label{relation1}
\end{align}
where $c_{\ell,n,-1}^k=0$.
 \end{thm}

\begin{defn} [$k$-linear succession]
For $\si\in G_{\ell,n}$,  the value
 $|\pi(i)|$ $(2 \leq  i \leq n)$  is a
 $k$-linear succession $(k\geq 1)$
of $\si$ at position  $i$ if $\pi(i)=\pi(i-1)+k$.
\end{defn}
Denote by  $\mathcal{L}^k(\si)$ the set of $k$-linear successions of $\si$
and let
 $l^k(\si)= \# \mathcal{L}^k(\si)$.
Let $l_{\ell,n,m}^k$ be the number of colored permutations in
$G_{\ell,n}$  with
 $m$ $k$-linear successions.
For example, 9 and  3 are the two 2-linear successions of the
permutation
$\si=\bar{5}\,\bar{2} \,4 \,7 \,9 \,\bar{1} \,\bar{3}
\,\bar{\bar{8}} \,6\in G_{4,9}$.

\begin{defn} [Skew $k$-linear succession]
For $\si=(\varepsilon,\sigma) \in G_{\ell,n}$, the value
 $\sigma(i)$     $(1 \leq  i \leq n) $   is a
 skew $k$-linear succession $(k\geq 1)$
of $\si$ at position  $i$ if
$$
\pi(i)=\pi(i-1)+k,
$$
where, by convention, $\sigma(0)=0$ and $\varepsilon(0)=1$.
\end{defn}
 Denote
by  $\mathcal{L}^{*k}(\si)$ the set of skew $k$-linear successions of $\si$
and
 $l^{*k}(\si)= \# \mathcal{L}^{*k}(\si)$. The number of permutations in $G_{\ell,n}$
 with
 $m$ skew $k$-linear successions is  $l_{\ell,n,m}^{*k}$.
Obviously  we have the following relation:
 \begin{align}\label{type1 type2}
\mathcal{L}^{*k}(\si)=\left\{%
\begin{array}{ll}
   \mathcal{L}^{k}(\si), & \hbox{if $\si(1)\not= k$;} \\
 \mathcal{L}^{k}(\si)\cup\{k\}, & \hbox{otherwise.} \\
\end{array}%
\right.
\end{align}

Let $\phid$ be the bijection  from  $G_{\ell,n}$ onto itself
defined by:
\begin{align}\label{defn:delta}
\si=\si_1\;\si_2 \; \cdots\, \si_n   \longmapsto \delta(\si)=\si_n
\; \si_1\;\si_2  \;\cdots \; \si_{n-1}.
\end{align}

\begin{thm}\label{phib} For any integer $k\geq 0$ there is a bijection
$\Phi$ from $G_{\ell,n}$ onto itself such that for $ \pi\in
G_{\ell,n}$,
\begin{equation}\label{prop1}
  \mathcal{C}^{k+1}(\si)=\mathcal{L}^{k+1}( \Phi(\pi) ),
  \end{equation}
and
\begin{equation}\label{prop0}
\mathcal{C}^k(\delta(\si))=\mathcal{L}^{*(k+1)}(\phib(\si)).
  \end{equation}
\end{thm}

Thanks to the transformation   $\Phi$
 the two statistics  $c^k$ and  $l^k$ are
 equidistributed on the group $G_{\ell,n}$ for $k\geq 1$. So we can replace
 the left-hand sides of (2.1)  by
$l_{\ell,n+1,m}^{k+1}$
and derive the following interesting result.
\begin{cor} \label{centreblc}
Let $ n$, $k$ and $m$ be integers such that   $n\geq 1$, $k \geq 0$
 and  $m \geq 0$. Then
\begin{align}\label{relationcentreblc}
l_{\ell,n+1,m}^{k+1}&=c_{\ell,n+1,m}^k+c_{\ell,n,m}^k-c_{\ell,n,m-1}^k,
\end{align}
where $c_{\ell,n,-1}^k=0$.
\end{cor}

Our proof of the last  two theorems is a generalization of that given by Clarke et al~\cite{CHZ},
where the  $(k,\ell)=(0,1)$ case of Corollary~\ref{centreblc}
is proved. Note that the $(k,\ell,m)=(0,2,0)$
case of \eqref{relationcentreblc} is the main result of a recent
paper by Chen and Zhang~\cite{CZ}.

\medskip

In order to interpret the entry $d_{\ell,n}^m$
we need the following definition.

\begin{defn} For $0\leq m\leq n$,
a permutation $\pi$ in $G_{\ell,n}$ is called
$m$-increasing-fixed
 if it satisfies the following conditions:
\begin{itemize}
\item[i)]  $\forall i\in [m]$, $\sign_{\si}(|\pi|(i))=1$;
\item[ii)] $\Fix(\si) \subseteq [m]$;
\item[iii)]$\pi(1)<\pi(2)<\cdots <\pi(m)$.
\end{itemize}
\end{defn}
Let $I_{\ell,n}^m$ be the set of  $m$-increasing-fixed permutations in
$G_{\ell,n}$. For example,
$$
I_{2,3}^2=\{
1\,2\,\bar 3,\quad 1\,3\,2,\quad 1\,3\,\bar 2,\quad 2\,3\,1,\quad 2\,3\,\bar 1\}.
$$

\begin{thm}\label{kbis} For $0\leq m\leq n$,
the entry $d_{\ell,n}^m$ equals the cardinality of $I_{\ell,n}^m$.
\end{thm}
\begin{proof}
Let $F_{\ell,n}^m$ be the set
of permutations with fixed points included in $[m]$ in
$G_{\ell,n}$. By Theorem~2 the cardinality of $F_{\ell,n}^m$ equals
$g_{\ell,n}^m$.
We define a mapping $f: (\tau,  \pi)\mapsto \tau\odot\pi$
from  $G_{\ell,m}\times F_{\ell,n}^m$ to
$F_{\ell,n}^m$ as follows:
$$
\tau\odot  \pi=\pi(\tau^{-1}(1))\pi(\tau^{-1}(2))\ldots \pi(\tau^{-1}(m))\pi(m+1)\ldots \pi(n).
$$
Clearly $f$ defines a group action of $G_{\ell,m}$ on the set
$F_{\ell,n}^m$.  We can choose an element
$\pi$ in each orbit such that
$$
\forall i\in [m],\quad \sign_{\si}(|\pi|(i))=1\quad \textrm{and}\quad
\pi(1)<\pi(2)<\cdots <\pi(m).
$$
 As the cardinality of the group $G_{\ell,m}$ is $\ell^m m!$,
 we derive that the number of the orbits
equals
 $g_{\ell,n}^m/\ell^m m!$.
\end{proof}

Rakotondrajao~\cite{Ra2} gave a different interpretation for
$d_{\ell,n}^m$ when $\ell=1$.
We can generalize her result
as in the following theorem.

\begin{defn} For $0\leq m\leq n$,
a permutation $\pi$
 in $G_{\ell,n}$ is called $m$-isolated-fixed if it satisfies the following conditions:
\begin{itemize}
\item[i)]   $\forall i\in [m]$, $\sign_{\si}(i)=1$;
\item[ii)] $\Fix(\si) \subseteq [m]$;
\item[iii)] each cycle of $\pi$ has at most one point in common with $[m]$.
\end{itemize}
\end{defn}
Let $D_{\ell,n}^m$ be the set of $m$-isolated-fixed permutations in $G_{\ell,n}$.
For example,
 $$
 D_{2,3}^2=\{
(1)(2)(\bar 3),\,(1,3)(2),\,(1,\bar 3)(2),\,
(1)(2,3),\, (1)(2,\bar 3)\}.
 $$
Note that $\pi=\bar 3\;1\;2\notin  D_{2,3}^2$ because 1 and 2 are in the same cycle.
\begin{thm}\label{k-fixe-point} For $0\leq m\leq n$,
the entry $d_{\ell,n}^m$ equals the cardinality  of $D_{\ell,n}^m$.
\end{thm}

As we will show in Section~7
there are more recurrence relations for
$g_{\ell,n}^m$ and $d_{\ell,n}^m$.
 In particular, we shall prove 
 an explicit formula for the
$\ell$-derangement numbers:
\begin{align}\label{formuleD}
d_{\ell,n}^0=g_{\ell,n}^0=n!\sum_{i=0}^n\frac{(-1)^i\ell^{n-i}}{i!},
\end{align}
which implies immediately the following recurrence relation:
\begin{align}\label{drec4}
d_{\ell,n}^0&= \ell nd_{\ell,n-1}^{0}+
(-1)^n \qquad (n\geq 1).
\end{align}

Note that \eqref{drec4} is the $\ell$-version of a famous recurrence for derangements. 
Using the combinatorial interpretation
for $g_{\ell,n}^m$ and $d_{\ell,n}^m$ it is possible to derive bijective proofs
of these recurrence relations. However
we  will just give combinatorial proofs for \eqref{drec4}  and two other recurrences
by generalizing the combinatorial proofs of Rakotondrajao~\cite{Ra2} for $\ell=1$ case,
and leave the others for the interested
readers.

The rest of this paper is organized as follows:  The proofs of
 Theorems~2, 3, 6 and 11 will be given in Sections~3, 4, 5 and
6, respectively. In
Section~7 we give the generating function of the coefficients
$g_{\ell,n}^m$'s and derive  more recurrence relations for the
coefficients $g_{\ell,n}^m$'s and $d_{\ell,n}^m$'s.
Finally, in Section~8
we give combinatorial proofs of three remarkable recurrence relations of $d_{\ell,n}^m$'s.

\section{Proof of Theorem~\ref{succession majore}}
Let  $m$ and $k$ be integers such that $n \geq m \geq k\geq 0$.
Denote by $G_{\ell,n}^m(k)$ the set of permutations in $ G_{\ell,n}$
whose $k$-circular successions are bounded by
 $m$ and $s_{\ell,n}^m=\#G_{\ell,n}^m(k)$.
We show that the sequence ($s_{\ell,n}^m$) satisfies  \eqref{eulerB}.

By definition, we have immediately
$G_{\ell,n}^n(k)=G_{\ell,n}$ and then  $s_{\ell,n}^n=\ell^n\,n!$.
Now, suppose $m<n$, then $G_{\ell,n}^{m+1}(k)\setminus G_{\ell,n}^m(k)$
is the set of permutations in $G_{\ell,n}^{m+1}(k)$ whose maximal
$k$-circular succession is $m+1$.
 It remains to show that  the cardinality of
 the latter  set equals $s_{\ell,n-1}^{m}$.
 To this end, we define a  simple bijection
 $\rho: \si\mapsto \pi'$ from
 $G_{\ell,n}^{m+1}(k)\setminus G_{\ell,n}^m(k)$ to $G_{\ell, n-1}^m(k)$
as follows.

 Starting from any $\pi=\pi_1\pi_2\ldots \pi_n$
in $G_{\ell,n}^{m+1}(k)\setminus G_{\ell,n}^m(k)$, we construct  $\pi'$ by deleting
$\pi_{m+1-k}=m+1$ and replacing each letter $\pi_i$ by $\pi_i-1$
if $|\pi_i|>m+1$. Conversely, starting from
$\si'=\pi_1'\pi_2'\ldots \pi_{n-1}'$ in ${G}_{\ell,n-1}^{m}(k)$,
 one can recover $\si$ by inserting
$m+1$ between $\pi_{m-k}'$ and $\pi_{m-k+1}'$ and then replacing
each letter $\pi'_{i}$ by $\pi_i'+1$ if $|\pi_i'|>m$.
For example, if
$\si=
3\, \bar{9}\, 5\, \bar{\bar{8}}\, \bar{7}\, \bar{\bar{6}}\, 2\, \bar{1}\,
4\in G_{3,9}^5(2)$, then
$ \si'=
3\,\bar{8}\, \bar{\bar{7}}\,\bar{6}\, \bar{\bar{5}}\, 2\, \bar{1}\, 4
\in {G}_{3,8}^4(2)$.
Note that $\pi' $ has a $k$-circular succession  $j\geq
m+1$ if and only if $j+1\geq m+2$ is a $k$-circular succession
 of $\pi$. Therefore, the maximal $k$-circular succession
of $\pi$ is $m+1$ if and only if the $k$-circular
successions of $\pi'$ are bounded by $m$. This completes the proof.

\begin{rmk}
The above argument does not explain why
 $g_{\ell,n}^m$ is independent from $k$ $(0\leq k\leq m)$.
We can provide such an argument as follows.
Consider the following simple bijection
$d$ which consists in  transforming
 $\pi=\pi_1\pi_2\pi_3\cdots\pi_n$ into
  $d(\pi)=\pi'=\pi_2\pi_3\cdots\pi_n\pi_1$.
  Clearly
  the  $k$-successions of $\pi$ are bounded by  $m$ if and only if
   the $(k+1)$-successions of
    $\pi'$ are bounded by
     $m$. Hence, denoting by $d^j$ the  composition of $j$-times of $d$,
     the application of $d^{k_2-k_1}$ permits to pass  from
      $k_1$-successions to  $k_2$-successions if $k_1<k_2$. In particular
      if we apply $m$ times  the mapping  $d$ to  a
       permutation whose fixed points are bounded  by
        $m$ then we obtain a  permutation without  $m$-succession and vice versa.
\end{rmk}
\section{Proof of Theorem~\ref{thm pol}}
Let $S_n^k(x)$ be the counting polynomial of the statistic $c^k$
 on the group $G_{\ell,n}$, i.e.,
\begin{align}
S_n^k(x)=\sum_{\si\in
G_{\ell,n}}x^{c^k(\si)}=\sum_{m=0}^nc_{\ell,n,m}^kx^m.
\end{align}
Then \eqref{relation1} is equivalent to
the following equation:
\begin{equation}\label{relation pol}
S_{n+1}^{k+1}(x)=S_{n+1}^{k}(x)+(1-x)S_{n}^{k}(x).
\end{equation}
By \eqref{defn:delta} it is readily seen that
\begin{align}\label{delta}
C^{k+1}(\si)=\left\{%
\begin{array}{ll}
   C^{k}(\delta(\si)), & \hbox{if $\si_n\not= k+1$;} \\
 C^{k}(\delta(\si))\setminus\{k+1\}, & \hbox{otherwise.} \\
\end{array}%
\right.
\end{align}
It follows that
\begin{align}
S_{n+1}^{k+1}(x) &=\sum_{\si\in G_{\ell,n+1}\atop
\si(1)=k+1}x^{c^k(\si)-1}
+\sum_{\si\in G_{\ell,n+1}\atop \si(1)\not=k+1}x^{c^k(\si)}\nonumber\\
&=\sum_{\si\in G_{\ell,n+1}\atop \si(1)=k+1}x^{c^k(\si)-1}
+\sum_{\si\in G_{\ell,n+1}}x^{c^k(\si)}-\sum_{\si\in
G_{\ell,n+1}\atop \si(1)=k+1}x^{c^k(\si)}.\label{eqrec}
\end{align}
For any $\pi\in G_{\ell,n+1}$ such that  $\si(1)=k+1$ we can
associate bijectively a permutation $\pi'\in G_{\ell,n}$ such that
$c^k(\pi)=c^k(\pi')+1$ as follows: $\forall i\in [n]$,
$$
\pi'(i)=\left\{%
\begin{array}{ll}
    \pi(i+1), & \hbox{if $\pi(i+1)\leq k$;} \\
    \pi(i+1)-1, & \hbox{if $\pi(i+1)>k$.} \\
\end{array}%
\right.
$$
Therefore we can rewrite \eqref{eqrec} as \eqref{relation pol}.

 We can also derive Theorem~3 from Theorem~2. First we prove a lemma.

\begin{lem} For $0\leq k\leq n-m$ there holds
\begin{align}\label{petitresultat}
c_{\ell,n,m}^k={n-k\choose m}g_{\ell,n-m}^k.
\end{align}
\end{lem}
\begin{proof}
To construct a permutation  $\pi$ in $G_{\ell,n}$ with $m$ $k$-circular successions
we can first choose $m$ positions $i_1,\ldots, i_{m}$ of $k$-circular successions among the first $n-k$ ones and then construct a
permutation $\pi_{0}$ of order $n-m$ without $k$-circular successions
on the remaining $n-m$ positions, where and in what follows we shall assume that 
$i_{1}<i_{2}<\cdots <i_{m}$.
 More precisely, there is  a bijection  $\theta:\, \pi \mapsto (I,\pi_0)$, where
 $I=\{i_1,\ldots, i_{m}\}$, from the set  of the colored permutations 
 of order $n$ with $m$ $k$-successions to the product of 
 the set of all $m$-subsets   of $[n-k]$ and the set of colored permutations
  of order $n-m$ without $k$-circular successions.
  
Denote by $G_{\ell,n,k,i}$ the set of all permutations in $G_{\ell,n}$ whose maximal position of
 $k$-circular successions equals $i$.  Define the mapping  $R_i: \pi \mapsto \pi'$   from $G_{\ell,n,k,i}$ to $G_{\ell,n-1}$ such that the linear form of $\pi'$ is obtained from $\pi=\pi_{1}\ldots \pi_{n}$ by removing the letter $(i+k)$ and replacing each colored letter $\pi_j$ by  $\pi_j-1$ if $|\pi_j|>i+k$.
 It is readily seen that the map $R_i$ is a bijection and  $c^k(\pi')=c^k(\pi)-1$. Indeed it is easy to see that $j+k$ is a $k$-circular succession of $\pi$ different of $i+k$ if and only if $j+k$ is a $k$-circular succession of $\pi'$. Hence, $\pi_0=R_{i_1}\circ R_{i_2}\circ \cdots\circ R_{i_m}(\pi)$ is a colored permutation
without $k$-circular succession in $G_{\ell,n-m}$.

Conversely given a subset $I=\{i_1,i_2,\cdots,i_m\}$ of $[n-k]$ and a  colored permutation $\pi_0$
without $k$-circular succession in $G_{\ell,n-m}$ we can  construct
$$
\pi=\theta^{-1}(I,\pi_0)=R_{i_m}^{-1}\circ R_{i_{m-1}}^{-1}\circ \cdots\circ R_{i_1}^{-1}(\pi_0),
$$
where $R_i^{-1}(\pi')$
is obtained  from $\pi'=\pi_{1}'\ldots \pi_{n-1}'$ by inserting the integer $(i+k)$ between 
$\pi_{i-1}'$ and $\pi_{i}'$ and  replacing each colored letter $\pi'_j$
by  $\pi'_j+1$ if $|\pi'_j|\geq i+k$. Therefore
\begin{align}\label{eq:pre}
c_{\ell,n,m}^k={n-k\choose m}c_{\ell,n-m,0}^k.
\end{align}
By Theorem~2 (ii) we have $c_{\ell,n,0}^k=g_{\ell, n}^k$.
Substituting this
in \eqref{eq:pre} yields then \eqref{petitresultat}.

\end{proof}

 By \eqref{petitresultat} we see that
\eqref{relation1} is equivalent to
$$
{n-k\choose m}g_{\ell, n+1-m}^{k+1}={n+1-k\choose m}g_{\ell,
n+1-m}^{k}+{n-k\choose m}g_{\ell, n-m}^{k}-{n-k\choose
m-1}g_{\ell, n+1-m}^{k}.
$$
Since $g_{\ell, n+1-m}^{k+1}-g_{\ell, n-m}^{k}=g_{\ell,
n-m+1}^{k}$ by \eqref{eulerB}, we can rewrite the last equation as
$$
{n-k\choose m}g_{\ell, n+1-m}^{k}={n+1-k\choose m}g_{\ell,
n+1-m}^{k}-{n-k\choose m-1}g_{\ell, n+1-m}^{k},
$$
which is obvious in view of the identity ${n-k\choose
m}={n+1-k\choose m}-{n-k\choose m-1}$. This completes the proof of
Theorem~3.

\section{Proof of Theorem~\ref{phib}}
There is a well-known bijection on the symmetric groups
transforming the cyclic structure into linear structure (see
\cite{FS} and \cite[p. 17]{St}).
 We need a variant of this  transformation, say $\varphi: S_n\to
 S_n$, as follows.

Given a permutation $\sigma\in S_n$ written  as a product of
cycles,  arrange the cycles in the decreasing order of their
maximum elements from left to right with the maximum element at
the end of each cycle. We then obtain $\varphi(\sigma)$ by erasing
the parentheses. Conversely, starting from a permutation written
in one-line form $\sigma'=a_1a_2\ldots a_n$, find out the
\emph{right-to-left maxima} of $\sigma'$ from right to left and
decompose the word  $\sigma'$ into blocks by putting a bar  at the
right of each right-to-left maximum and construct a cycle of
$\sigma$ with each block.

For example, if $\sigma=(3,\, 1,\, 4, \,6,\, 9)(5, \,7, \,8)(2)\in S_9$
then $\varphi(\sigma)=3\;1\; 4\, 6\; 9\; 5\; 7\; 8\; 2$.
Conversely, starting from $\sigma'=3 \;1\; 4\, 6\; \textbf{9}\;
5\; 7\; \textbf{8}\; \mathbf{2}$, so the right-to-left maxima are
9,8 and 2, then the decomposition into blocks is   $3\; 1\; 4
\;6\; 9| 5 \;7 \;8|2|$ and we recover $\sigma$ by putting
parentheses around each block.
\begin{lem}\label{foata}
For $k\geq 1$, the mapping $\varphi$ transforms
the  {\kscs} to  {\ksls} and  vice versa.
\end{lem}
\begin{proof}
Indeed,
an integer $p$ is a {\ksc}  of $\sigma$  \sii  there is an integer
$i \in [n]$ such that
 $\sigma(i)=i+k$, so  $i$ and $i+k$ are two consecutive letters in the one-line form
 of $\varphi(\sigma)$. Conversely if $i$ and $i+k$ are two consecutive letters
 in the one-line form of a permutation $\tau$ then $i$ cannot be a right-to-left maximum, so
 $i$ and $i+k$ are in the same cycle of $\varphi^{-1}(\tau)$, say $\sigma$,
  and then $\sigma(i)=i+k$.

\end{proof}

 We now construct a
bijection $\Phi: \pi\mapsto \pi'$  from $G_{\ell,n}$ onto itself
such that
\begin{equation}
  C^k(\si)=L^k( \pi') \quad  (k \geq 1).
  \end{equation}
  Let
$\sigma=|\pi|$ and $\sigma'=|\pi'|$. \\
\noindent{\bf Bijection $\Phi$:} First define
$\sigma'=\varphi(\sigma)$: Factorize $\sigma$ as product of $r$
disjoint cycles $C_1,\ldots, C_r$.  Suppose that $\ell_i$ and $g_i$
are, respectively,  the length and
 greatest
element of the  cycle $C_i$ ($1\leq i\leq r$)  such that
$g_1>g_2>\cdots >g_r$. Then
\begin{align}\label{linear}
\sigma'=\sigma(g_1)\;\cdots\;
\sigma^{\ell_1-1}(g_1)\;g_1\;\sigma(g_2)\;\cdots\;
\sigma^{\ell_2-1}(g_2)\;g_2\;\cdots \;\sigma(g_r)\;\cdots\;
\sigma^{\ell_r-1}(g_r)\;g_r.
\end{align}
Let  $T_\sigma =\lbrace  \sigma(g_i), \  i \in [r]\rbrace$.
It remains to define $\sign_{\pi'}(\sigma^j(g_i))$ for all $i\in
[r]$ and $1\leq j\leq \ell_i$.
 We proceed by induction on $j$ as
 follows: For each $i\in [r]$ let
 $$
\sign_{\si'}(\sigma(g_i))=\sign_{\pi}(\sigma(g_i)),
 $$
 and for $j=2,\ldots, \ell_j$ define
\begin{equation}\label{sign}
\sign_{\si '}(\sigma^{j}(g_i))= \sign_{\si
'}(\sigma^{j-1}(g_i))\cdot \sign_{\si}(\sigma^j(g_i)).
\end{equation}
It is easy to establish the inverse of $\Phi$.

 \noindent{\bf
Bijection $\Phi^{-1}$:} Starting from $\pi'$ we can recover
$\sigma=|\pi|$ by applying $\varphi^{-1}$ to $\sigma'$. Suppose
$\sigma'$ is given as in \eqref{linear} and $g_1,\ldots, g_r$ are
the left-to-right-maxima. We then determine
$\sign_{\pi}(\sigma^j(g_i))$ for all $i\in [r]$ and $1\leq j\leq
\ell_i$ as
 follows: For each $i\in [r]$ let
 $$
\sign_{\si}(\sigma(g_i))=\sign_{\pi'}(\sigma(g_i)),
 $$
 and for $j=2,\ldots, \ell_j$ define
\begin{equation}\label{sign'}
\sign_{\si }(\sigma^{j}(g_i))=
\sign_{\si'}(\sigma^j(g_i))(\sign_{\si
'}(\sigma^{j-1}(g_i)))^{-1},
\end{equation}
where $\sign_{\pi}^{-1}(i)$ is the inverse of $\sign_{\pi}(i)$ in
the cyclic group $C_{\ell}$.

\begin{lemma} For $k\geq 1$,
the mapping $\Phi$ transforms a $k$-circular succession of $\si$
to a  $k$-linear succession of $\si'$ and  vice versa.
\end{lemma}
\begin{proof}
By Lemma~\ref{foata} we have the equivalence: $\sigma(i) $ is a
$k$-circular succession of $\sigma$ if and only if $\sigma(i)$ is
a $k$-linear succession of $\sigma'$. It remains to  verify that
if $\sigma(i)$ is a $k$-circular succession
 of $\sigma$ then
\begin{align}\label{eqH}
\sign_{\si}(\sigma(i))=1 \Leftrightarrow
\sign_{\si'}(\sigma(i))=\sign_{\si'}(i)
 \end{align}
Note that if $\sigma(i)$ is a $k$-circular succession of $\sigma$
then  $i$ and $\sigma(i)$ must be in the same cycle
 and that $\sigma(i)$ cannot be in  $T_\sigma$ for, otherwise, $i$ would be the greatest element
 of the cycle but this is impossible because
  $\sigma(i)=i+k$ ($k\geq 1$).

 Now,  assume that $\sigma(i)$ is a  $k$-circular succession of $\sigma$.
\begin{itemize}
\item[(i)] Suppose $\sign_{\si}(\sigma(i))=1$. As $\sigma(i)\not \in
T_\sigma $ and $\sigma^{-1}(\sigma(i))=i$, we have
$$
\sign_{\si'}(\sigma(i))=\sign_{\si'}(i)\cdot
\sign_{\si}(\sigma(i))=\sign_{\si'}(i).
$$
\item[(ii)] Suppose that
$\sign_{\si'}(\sigma(i))=\sign_{\si'}(i)$. As $\sigma(i)\not \in
T_\sigma $ and $\sigma^{-1}(\sigma(i))=i$, we have
$$
\sign_{\si}(\sigma(i))=\sign_{\si'}(\sigma(i))/\sign_{\si'}(i)=1.
 $$
\end{itemize}
Hence \eqref{eqH} is established.
\end{proof}

Obviously the above lemma is equivalent to \eqref{prop1}. We
obtain \eqref{prop0} by combining  \eqref{prop1}, \eqref{type1
type2} and \eqref{delta}.

 We conclude this section with an
example. Consider
$$
\si=  \left (
\begin{array}{lllllllll}
1&2&3&4&5&6&7&8&9\\
\bar{3}&4&\bar{9}&\bar{8}&7&\bar{5}&6&\bar{\bar{2}}&\bar{\bar{1}}
\end{array}
\right) \in G_{4,9}.
$$

Factorizing  $\sigma=|\pi|$ into cycles we get  $\sigma=(1, 3 ,9 ) (
2 ,4, 8) (6, 5, 7 )$, then
$$\sigma'= 1 \ 3 \ 9 \ 2\  4\  8\ 6\  5  \ 7
\quad \text{and}\quad T_\sigma= \lbrace 1, 2, 6\rbrace.
$$
 The  signs of $\sigma'(i)$ for  $i$  $ \in [9]$ are computed as
 follows:
\begin{align*}
\sign_{\si '}(1)&=\sign_{\si}(1)=\zeta^2\quad \textrm{for}\quad 1 \in T_\sigma;\\
\sign_{\si '}(3)&=\sign_{\si '}(1)\cdot \sign_{\si}(3)=\zeta^2.\zeta=\zeta^3\quad \textrm{for}\quad 3 \not\in T_\sigma;\\
\sign_{\si '}(9)&=\sign_{\si '}(3)\cdot \sign_{\si}(9)=\zeta^3.\zeta=1\quad \textrm{for}\quad 9\not\in T_\sigma;\\
\sign_{\si '}(2)&=\sign_{\si '}(2)=\zeta^2 \quad \textrm{for}\quad 2 \in T_\sigma;\\
\sign_{\si '}(4)&=\sign_{\si '}(2)\cdot \sign_{\si }(4)=\zeta^2\cdot1=\zeta^2\quad \textrm{for}\quad 4 \not\in T_\sigma;\\
\sign_{\si '}(8)&=\sign_{\si '}(4)\cdot \sign_{\si }(8)=\zeta^2\cdot\zeta =\zeta^3\quad \textrm{for}\quad 8 \not\in T_\sigma;\\
\sign_{\si '}(6)&=\sign_{\si}(6)=1\quad \textrm{for}\quad 6 \in T_\sigma;\\
\sign_{\si '}(5)&=\sign_{\si '}(6)\cdot \sign_{\si}(5)=1\cdot\zeta=\zeta\quad \textrm{for}\quad 8 \not\in T_\sigma;\\
\sign_{\si '}(7)&=\sign_{\si '}(5)\cdot \sign_{\si}(7)=\zeta\cdot 1=\zeta\quad \textrm{for}\quad 7 \not\in T_\sigma.
\end{align*}
Thus we have
$\pi\mapsto \si'=
\bar{\bar{1}} \,\bar{\bar{\bar{3}}} \,9 \,\bar{\bar{2}}\,
\bar{\bar{4}}\,\bar{\bar{\bar{8}}} \,6 \,\bar{5} \,\bar{7}
$. We have
$C^2(\pi)=\{4,7\}=L^{*2}(\pi')$.

Conversely, starting from $\pi'$, we can recover $\sigma$ by
$\varphi^{-1}$ and the signs of $\sigma(i)$ ($i\in [n]$) by
\eqref{sign'}.  As $\sigma = (1 3 9 ) ( 2 4 8) (6 5 7 )$ and  $T_\sigma =
\lbrace 1, 2, 6\rbrace$, we have, for example,
$$
\sign_{\si}(9)=\sign_{\si'}(9)\cdot
\sign_{\si'}^{-1}(3)=1\cdot\zeta=\zeta
$$
for $9\not\in T_\sigma $.

\section{Proof of Theorem~\ref{k-fixe-point} }
We shall give two proofs by using Theorems~9 and 2, respectively.
\subsection{First Proof}
We shall define
a mapping $\varphi: \pi \mapsto \pi'$
from $D_{\ell,n}^m$ to $I_{\ell,n}^m$  in two steps. 
First we establish 
the correspondence   $|\pi|\mapsto |\pi'|$ and then 
determine the  sign  transformation. 
Define the permutation $|\pi'|=|\pi'|(1)\ldots |\pi'|(n)$ such that
$|\pi'|(1)\ldots |\pi'|(m)$ is  the increasing rearrangement of
$|\pi|(1),\ldots, |\pi|(i_m)$
and $|\pi'|(m+1)\ldots |\pi'|(n)=|\pi|(m+1)\ldots |\pi|(n)$.
Conversely, starting from $\pi'\in I_{\ell,n}^m$,
for each $i\in [m]$ we 
construct the cycle
of $|\pi|$  containing $i$ by
$$
(|\pi'|^{-s}(i), \ldots, |\pi'|^{-2}(i), |\pi'|^{-1}(i),i),
$$
where $s$ is the smallest non negative integer such that $|\pi'|^{-s}(i)\in T$, where
$$T :=\{|\pi'|(i),i\in[m] \},
$$
 and by convention $|\pi'|^0(i)=i$. 
 In particular if $i\in [m]\cap T$, then $s=0$ and $i$ is a fixed point of $|\pi|$.
The other cycles remain unaltered. 

For
example, for
$\pi=(1)(2,\bar{7},\bar{\bar{6}})(3,\bar{\bar{5}},9)(4)(\bar{\bar{8}}) \in D_{3,9}^4$ 
(i.e., $n=9$, $\ell =3$, $m=4$), we have 
$|\pi|=(1)(2,7,6)(3,5,9)(4)(8)$ and $|\pi'|=145792683=(1)(2476)(359)(8)$,  so $T=\{1,4,5,7\}$.

Now, we describe  the sign transformation.  For each $i\in [m]$, 
since the letter $i$ of $\pi $ ($m$-isolated-fixed) as well as the letter $\pi(i) $  of $\pi' $ ($m$-increasing-fixed) are uncolored,  the transformation of the signs is obtained by exchanging the sign of $| \pi|(i) $ 
and that of $i$, namely
$$ \sign_{\pi'}(i)=\sign_{\pi}(|\pi|(i))\quad  \text{ and } \quad 
\sign_{\pi}(i)=\sign_{\pi'}(|\pi|(i))=1\quad \forall i \in [m];
$$
the signs of other letters reamain unaltered, i.e.,
$$\sign_{\pi'}(i)=\sign_{\pi}(i) \qquad \forall i \in [n] \setminus (T \cup [m]).
$$

Continuing the above example, we have $\sign_{\pi'}(2)=\sign_{\pi}(\pi(2))=\sign_{\pi}(7)=\zeta^1$; 
$\sign_{\pi'}(3)=\sign_{\pi}(\pi(3))=\sign_{\pi}(5)=\zeta^2$,  hence
$\pi'=1\,4\,5\,7\,9\,\bar{2}\,\bar{\bar{6}}\,\bar{\bar{8}}\,\bar{\bar{3}}$.

\subsection{Second proof }
Let $G_{\ell,n}^m:=G_{\ell,n}^m(0)$ be the set
of permutations in $G_{\ell,n}$ whose fixed points are included in $[m]$.

Let $\pi=(\varepsilon, \sigma)$ be  a permutation in $G_{\ell,n}^m$,
written as a product of disjoint cycles. For each $i\in [m]$,
  let $s$ be the smallest integer $\geq 1$ such that $\sigma^s(i)\in [m]$ and
$w_{\pi}(i)=[\varepsilon_{\sigma(i)}\sigma(i)]\ldots [\varepsilon_{\sigma^{s-1}(i)}\sigma^{s-1}(i)]$.
 Clearly $w_{\pi}(i)=\emptyset$ if $s=1$. Let $\Omega_\pi$ be the product of  cycles of $\pi$
  which have no common point with $\{i\zeta^j| i\in [m], 0\leq j\leq \ell-1\}$ and
  $\pi_m$ be the permutation in $G_{\ell,m}$ obtained from
  $\pi$ by deleting the cycles in $\Omega_\pi$
  and letters in $w_{\pi}(i)$ for $i\in [m]$.
 For example, if
 $\pi= (\bar{\bar{1}}\,\bar{4}\,7\,\bar{\bar{3}}\,2\,\bar{6}\,5)(\bar{8})(\bar{\bar{9}})
 \in G_{3,9}^3$ then $\pi_3=(\overline{\overline{1}}\,\overline{\overline{3}}\,2)$,
 $$w_{\pi}(1)=\bar{4}\,7,\quad w_{\pi}(2)=\bar{6}\,5,\quad
 w_{\pi}(3)=\varnothing \quad \textrm{and}\quad
\Omega_{\pi}=(\bar{8})\,(\bar{\bar{9}}).
$$
Let $T(\pi)=(w_{\pi}(1),w_{\pi}(2),\cdots,w_{\pi}(m),\Omega_{\pi})$ and define
the  relation $\sim$  on $G_{\ell,n}^m$ by
 $$
 \pi_1 \sim \pi_2 \Leftrightarrow T(\pi_1)=T(\pi_2).
 $$
 Clearly this is an equivalence relation. To determine the equivalence class ${\mathcal C}_\pi$
 of each permutation $\pi\in G_{\ell,n}^m$
 we consider the mapping $\theta: (\tau,\pi) \mapsto \theta(\tau, \pi)$ from
$G_{\ell,m}\times G_{\ell,n}^m$ to $G_{\ell,n}^m$,  where the cyclic factorization of
 $\theta(\tau,\pi)$ is obtained
by inserting the word  $w_{\pi}(i)$ after
each letter $i\zeta^j$  appearing in a cycle of
$\tau$  for each $i\in [m]$ and some $j: 0\leq j\leq \ell-1$, and then  add the cycles in $\Omega_\pi$.
For example, if
$\pi= (\bar{\bar{1}}\,\bar{4}\,7\bar{\bar{3}}\,2\bar{6}\,5)(\bar{8})(\bar{\bar{9}})\in G_{3,9}^3$
and  $\tau= (\bar{1}\,\bar{\bar{2}})(3)\in G_{3,3}$  then
$$\theta(\tau,\pi)=(\bar{1}\,\bar{4}\,7\,\bar{\bar{2}}\,\bar{6}\,5)(3)(\bar{8})(\bar{\bar{9}}).
$$
Clearly ${\mathcal C}_\pi=\{\theta(\tau,\pi)|\tau \in  G_{\ell,m}\}$. Indeed,
by definition $\theta(\tau,\pi)\sim \pi$ for each $\tau\in G_{\ell,m}$ and $\pi\in G_{\ell,n}^m$ and
conversely, if $\pi'\sim \pi$ then $\pi'=\theta(\pi_m',\pi)$
for $T(\pi')=T(\pi)$. Moreover, suppose $\theta(\tau,\pi)=\theta(\tau',\pi)=\pi'$
 for $\tau, \tau'\in G_{\ell,m}$, then
$\tau=\tau'=\pi'_m$. Hence the cardinality of each equivalence class is $\ell^mm!$ and, by Theorem~2
 the number of equivalence classes equals $d_{\ell,n}^m=g_{\ell,n}^m/\ell^mm!$.
Choosing  $\theta(\iota,\pi)$ as the representative
of the class  ${\mathcal C}_\pi$, where $\iota$ is the identity of $G_{\ell,m}$,
yields the desired result.

\section{Generating functions and further recurrence relations}
For any function $f:\Z\to \mathbb{C}$ introduce  the difference
operator $\Delta{f(n)}=f(n)-f(n-1)$. Then it is easy to see by
induction on $N\geq 0$ that
\begin{align}\label{fdiff}
\Delta^Nf(n)=\sum_{i=0}^N(-1)^i{N\choose
i}f(n-i)=\sum_{i=0}^N(-1)^{N-i} {N\choose i}f(n-N+i).
\end{align}

\begin{prop} For $m\geq 0$ the following identities hold true:
\begin{align}
g_{\ell,n+m}^m=\sum_{i=0}^n&(-1)^{n-i}{n\choose i}\ell^{m+i}(m+i)!,\label{derg}\\
\sum_{n \geq 0 }g_{\ell,n+m}^m \frac{u^n}{n!}&
=\frac{\ell^mm!\exp(-u)}{(1-\ell u)^{m+1}},\label{gfeq}\\
\sum_{m, n \geq 0 }g_{\ell,n+m}^m \frac{x^m}{m!}\frac{u^n}{n!}&=
\frac{\exp(-u)}{1-\ell x-\ell u}.\label{gfeq2}
\end{align}
\end{prop}
\begin{proof}

Setting $f(n)=g_{\ell,n}^{n}$ then
$g_{\ell,n+m}^{n+m-i}=\Delta^i{f(n+m)}$ for $i\geq 0$. It follows
from \eqref{fdiff} that
\begin{align}
g_{\ell,n+m}^m=\Delta^{n}{f(n+m)}=\sum_{i=0}^n(-1)^{n-i}{n\choose
i}\ell^{m+i}(m+i)!.
\end{align}
Multiplying the above identity by $u^n/n!$ and summing
 over $n\geq 0$  we obtain
\begin{align*}
\sum_{n \geq 0 }g_{\ell,n+m}^m \frac{u^n}{n!}=\ell^mm!\sum_{n,i\geq
0}(-1)^{n-i}{m+i\choose i}\frac{\ell^iu^n}{(n-i)!}.
\end{align*}
Shifting $n$ to $n+i$ yields
\begin{align*}
\sum_{n \geq 0 }g_{\ell,n+m}^m \frac{u^n}{n!}=\ell^mm!\left(\sum_{n\geq
0}(-1)^{n} \frac{u^{n}}{n!}\right) \cdot \left(\sum_{i\geq
0}{m+i\choose i}(\ell u)^i\right),
\end{align*}
which is clearly equal to the right-hand side of \eqref{gfeq}.
Finally multiplying  \eqref{gfeq} by  $x^m/m!$ and summing
over  $m\geq 0$ yields \eqref{gfeq2}.
\end{proof}

Setting  $m=0$ in \eqref{derg} yields  immediately formula \eqref{formuleD}.

\begin{prop} \label{prop-1}
For $\ell\geq 0$ and $0 \leq m \leq n$ there hold
\begin{align}
g_{\ell,n}^m&=(\ell n-1)g_{\ell,n-1}^m+\ell(n-m-1)g_{\ell,n-2}^m\qquad (n\geq 2);\label{rec1}\\
g_{\ell,n}^m&=\ell(n-m)g_{\ell,n-1}^m+\ell mg_{\ell,n-1}^{m-1}\qquad (m\geq 1,\; n\geq 1); \label{rec2}\\
g_{\ell,n}^m&=\ell ng_{\ell,n-1}^{m}-\ell mg_{\ell,n-2}^{m-1} \qquad (m\geq 1,\;
n\geq 2); \label{rec3}
\end{align}
where $g_{\ell,0}^0=1$, $g_{\ell,1}^0=\ell-1$ and
$g_{\ell,1}^1=\ell$.
\end{prop}
\begin{proof}
Let $F(u)$ be the left-hand side of \eqref{gfeq}. Differentiating
$F(u)$ and using the right-hand side of \eqref{gfeq}  we get
\begin{align}\label{keyeq}
(1-\ell u)F'(u)=[\ell (m+1)-1+\ell u]F(u).
\end{align}
Equating  the coefficients of $u^n/n!$ in \eqref{keyeq} yields
$$
g_{\ell,n+m+1}^m=[\ell(m+n+1)-1]g_{\ell,n+m}^m+\ell ng_{\ell,n+m-1}^m,
$$
which gives \eqref{rec1} by shifting  $n+m+1$ to $n$.

Next, multiplying the two sides of \eqref{gfeq} by $1-\ell u$ gives
\begin{align}
(1-\ell u)\sum_{n \geq 0 }g_{\ell,n+m}^m
\frac{u^n}{n!}=\frac{\ell^mm!\exp(-u)}{(1-\ell u)^{m}}=\ell m \sum_{n \geq 0
}g_{\ell,n+m-1}^{m-1} \frac{u^n}{n!}.
\end{align}
Equating the coefficients of $u^n/n!$ yields
\begin{align}
g_{\ell,n+m}^m-\ell ng_{\ell,n+m-1}^m=\ell mg_{\ell,n+m-1}^{m-1},
\end{align}
which is \eqref{rec2} by shifting $n+m$ to $n$.

Finally, we derive  \eqref{rec3} from \eqref{rec2} and
\eqref{eulerB}:
\begin{align*}
g_{\ell,n}^m&=\ell ng_{\ell,n-1}^{m}-\ell m(g_{\ell,n-1}^m-g_{\ell,n-1}^{m-1})\\
&=\ell ng_{\ell,n-1}^{m}-\ell mg_{\ell,n-2}^{m-1}.
\end{align*}
 The proof is thus completed.
\end{proof}

It is easy to convert the above relations for $g_{\ell,n}^m$ to
those for $d_{\ell,n}^m$.
\begin{prop} For $\ell\geq 0$ and $0\leq m \leq  n$ we have
\begin{align}
d_{\ell,n}^m&=(\ell n-1)d_{\ell,n-1}^m+\ell(n-m-1)d_{\ell,n-2}^m\qquad (n\geq 2);\label{drec1}\\
d_{\ell,n}^m&=\ell(n-m)d_{\ell,n-1}^m+d_{\ell,n-1}^{m-1}\qquad (m\geq 1,\; n\geq 1);\label{drec2}\\
d_{\ell,n}^m&+d_{\ell,n-2}^{m-1}=\ell nd_{\ell,n-1}^{m}\qquad (m\geq 1,\;
n\geq 2),\label{drec3}
\end{align}
where $d_{\ell,0}^0=1$, $d_{\ell,1}^0=\ell-1$ and
$d_{\ell,1}^1=1$.
\end{prop}
\begin{proof}
The equations \eqref{drec1}, \eqref{drec2} and \eqref{drec3}
follow directly from Proposition~\ref{prop-1}.
\end{proof}

\section{Combinatorial proofs of three recurrence relations}
Using the combinatorial interpretation for $d_{\ell,n}^n$ in Theorem~11
we now give combinatorial interpretations of \eqref{deulerB},  (\ref{drec4}) and \eqref{drec3}
by  generalizing the proofs of Rakotondrajao \cite{Ra2},
 which correspond to the $\ell=1$ case.
\subsection{Combinatorial proof of \eqref{deulerB}}
We shall prove that the cardinality of
$D_{\ell,n}^m$ satisfies the following recurrence:
\begin{align}
d_{\ell,n}^n=1\quad\textrm{and}\quad
d_{\ell,n}^{m-1}+d_{\ell,n-1}^{m-1}=\ell m\, d_{\ell,n}^{m}\qquad
(1\leq m\leq n).\label{deulerB'}
\end{align}
First, the identity permutation is the only
$n$-isolated-fixed permutation in
$G_{\ell,n}$, so $d_{\ell,n}^n=1$.
To prove \eqref{deulerB'} we construct a bijection
$\vartheta: \si
\longmapsto (\epsilon,\alpha,\si')$
from $ D_{\ell,n-1}^{m-1} \cup D_{\ell,n}^{m-1}$ to
 $\mathcal{C}_\ell \times [m] \times D_{\ell,n}^m $
as follows:

 Let
$\sigma=|\pi|$ and factorize $\pi$ into disjoint cycles.
\begin{itemize}
\item[1.] If $\si \in D_{\ell,n-1}^{m-1}$,
  then $\epsilon=1$, $\alpha=m$,
 the  cycles of $\si'$ are obtained from those of $\si$
 by  substituting
  $\zeta^ji$ by $\zeta^ji +1$  if $i \geq m$ and then adding the cycle $(m)$.
\item[2.] If $\si \in D_{\ell,n}^{m-1}$, let $C_m$ be the cycle of $\sigma$ containing
$m$. Then $\epsilon=\sign_{\si}(m)$ and $\alpha$ is the smallest
integer in $C_m$; let  $q$ be the smallest integer such that
$\sigma^q(m)=\alpha$; then $\sigma'$ is obtained from $\sigma$ by
deleting the letters $m,\,\sigma(m),\,\ldots,\,\sigma^{q-1}(m)$
from $C_m$ and creating a new cycle
 $(m\,\sigma(m)\,\cdots\,\sigma^{q-1}(m))$. Finally,
 define  the sign of  $i\in [n]$ in $\pi'$
by
$$
\sign_{\si'}(i)=\left\{%
\begin{array}{ll}
    \sign_{\si}(i), & \hbox{ if $i\not =m$;} \\
    1, & \hbox{if $i =m$.} \\
\end{array}%
\right.
$$
\end{itemize}
For example, let $\ell=3,\, n=9$ and $m=6$. If  $\si=
(1\bar{\bar{6}})(2)(3)(4)(5)(\bar{7}\bar{\bar{8}})\in D_{3,8}^5$,
then
$$\epsilon=1;\alpha=6 \quad{and}\quad
\si'=(1\bar{\bar{7}})(2)(3)(4)(5)(6)(\bar{8}\bar{\bar{9}});
$$
if $\si=(1)(2\bar{9}\bar{6}\bar{\bar{8}})(3)(4)(5)(\bar{7}) \in
D_{3,9}^5$ then
$$
\alpha=2, \quad \epsilon=\zeta\quad{and}\quad
\si'=(1)(2\bar{9})(3)(4)(5)(\bar{7})(6\bar{\bar{8}}).
$$

It remains to show that $\vartheta$ is a bijection.
Given $(\epsilon,\alpha,\si')$ let $\sigma'=|\pi'|$.
We define the inverse $\vartheta^{-1}: (\epsilon,m,\si')\mapsto \si$ as follows:
\begin{enumerate}
\item  If $\alpha=m$; $\epsilon=1$ and $\si'(m)=m$ then
the cycles of $\si$ are obtained by deleting the  cycle $(m)$ and replacing
 $\zeta^ji$ by $\zeta^ji -1$  if $i \geq m$ in $\pi'$.
\item  If $\alpha=m$; $\epsilon  = 1 $ and $\si'(m) \not =m$ then
$\pi=\pi'$.
\item  If $\alpha=m$;
$\epsilon \not = 1 $ then  we get  $\pi$ from $\pi'$ by replacing
$m$  by $\epsilon\, m$.

\item  If $\alpha<m$; $\sigma$ is obtained from $\sigma'$ by
removing the cycle which contains $m$ and then inserting
 the word $m\sigma'(m)\sigma'^2(m)...\sigma'^{q-1}(m)$, where $\sigma'^q(m)=m$,
  in the cycle which
 contains the integer $\alpha$ just before the integer $\alpha$.
In other words, we define $m=\sigma(\sigma'^{-1}(\alpha))$,
$\sigma^j(m)=\sigma'^{j}(m)$ for $1\leq j\leq q-1$ and
$\sigma^q(m)=\alpha$.
 Finally
 define $\sign_{\si}(i)
 =\sign_{\si'}(i)$ for $i\not=m$ and $\sign_{\si}(m)=\epsilon$.

\end{enumerate}

 To see that this is indeed the inverse of $\vartheta$ we just
 note the following simple facts:
\begin{itemize}
  \item If $\si \in D_{\ell,n-1}^{m-1}$  then $\epsilon=1$, $\alpha=m$ and $m \in
  FIX(\pi')$.  We have
  $\vartheta(\pi) \in E_1= \{(1,m,\pi'),  m \in FIX(\pi'), \pi' \in   D_{\ell,n}^{m}
  \}$.
  \item If $\pi \in D_{\ell,n}^{m-1}\bigcap D_{\ell,n}^{m}$
  then $\epsilon=1$, $\alpha=m$ and $m \not \in FIX(\pi')$. In this case $\pi'=\pi$
 we have
$\vartheta(\pi) \in E_2= \{(1,m,\pi') ,  m \not \in FIX(\pi'),
\pi' \in   D_{\ell,n}^{m} \}$.
  \item If $ m \not \in $ cycle of $\sigma$ containing $i<m$ and $\sign_{\pi}(m) \not = 1$
  then $\alpha=m$, $\epsilon  \not = 1$;  and $\pi'$ is obtained from $\pi$
  by just replacing $\epsilon m$ by $m$.  We have
       $\vartheta(\pi) \in E_3= \{(\epsilon,m,\pi')$,
       $\epsilon \not =1$, $\pi \in   D_{\ell,n}^{m} \}$.
  \item If $ m  \in $ cycle of  $\sigma$ containing $i<m$ then
  $\alpha<m$.
    In this case the image $\pi'$ is defined by the second case of the
     construction of $\vartheta$.
 We have $\vartheta(\pi) \in E_4=\mathcal{C}_\ell \times [m-1] \times
D_{\ell,n}^m $.
\end{itemize}
Clearly  $\{E_1, E_2,E_3,E_4\}$ is a  partition of $D_{\ell,n}^m $.

\subsection{Combinatorial proof of (\ref{drec4})}
We shall prove the following version of (\ref{drec4}):
\begin{gather*}
\ell n d_{\ell,n-1}^0-1=d_{\ell,n}^0\qquad \textrm{if $n$ is odd},\\
\ell n d_{\ell,n-1}^0=d_{\ell,n}^0-1\qquad \textrm{if $n$ is even}.
\end{gather*}

Denote by $\mathcal{D}_{\ell,n}$  the set of derangements in $G_{\ell,n}$. Let
$E_n=\emptyset$  if $n$ is odd and
$E_n=\{(1 \, \,2)(3 \, \, 4)\cdots(n-1 \,\, n) \}$ if $n$ is even.
Introduce also $F_n=\emptyset$ if $n$ is even and
$F_n=\{1\}\times \{n\}\times E_{n-1}$ if $n$ is odd.
We are going to define a mapping $\tau_\ell: (\varepsilon,k,\pi)\longmapsto \pi'$
from
$\left(\mathcal{C}_{\ell}\times [n]\times
\mathcal{D}_{\ell,n-1}\right)\setminus F_n$ to
$\mathcal{D}_{\ell,n} \setminus E_n$, which implies the above identities.

Factorize $\pi$ into disjoint cycles.
We construct the cyclic factorization of $\pi'$ by
distinguishing several cases and by giving an example  in
$G_{4,9}$ for each case.
Let $c(k)$ be
the length of the cycle  of $\pi$ containing $k$ and write $\widehat{k}=\sign_{\pi}(k)\cdot k$.

\begin{itemize}
  \item [1.] If $k < n $,
  we obtain $\pi'$ by inserting $\varepsilon \, n$ just after $\widehat{k}$ in a cycle of $\pi$.\\
Example: $\varepsilon=\zeta^3, \, k=3, \,
\pi=(\bar{\bar{1}} \, \bar{4} \, 2)(\bar{\bar{3}})(\bar{5} \,\bar{\bar{6}}\,8\, \bar{\bar{\bar{7}}}) $
then $\pi'=(\bar{\bar{1}} \, \bar{4} \, 2)
(\bar{\bar{\bar{9}}} \bar{\bar{3}})(\bar{5} \,\bar{\bar{6}}\,8\, \bar{\bar{\bar{7}}}) $.
  \item [2.] If $k = n $ and $\varepsilon \not= 1$,
 we obtain $\pi'$ by creating the cycle  $(\varepsilon n)$.\\
Example: $\varepsilon=\zeta^3, \, k=9, \, \pi=(\bar{\bar{1}} \, \bar{4} \, 2)
(\bar{\bar{3}})(\bar{5} \,\bar{\bar{6}}\,8\, \bar{\bar{\bar{7}}}) $ then $\pi'=(\bar{\bar{1}} \, \bar{4} \, 2)
( \bar{\bar{3}})(\bar{5} \,\bar{\bar{6}}\,8\, \bar{\bar{\bar{7}}})(\bar{\bar{\bar{9}}}) $.
\item [3.] Suppose  $k= n $ and $\varepsilon = 1$.
 Let $p\geq 0$ be the  smallest integer such that
the transposition
$(2p+1, \,2p+2)$ is not a cycle of
 $\pi$.
In all the examples of this part we take $p=2$.
 \begin{itemize}
        \item [3.1] If $\sign_{\pi}(2p+1)=1$ then
            \begin{itemize}
            \item [3.1.1] If $2p+2$ is
            a  1-circular succession of $|\pi|$
            then $\pi'$ is obtained by
            deleting $2p+1$ and  creating
            the cycle $(n, 2p+1)$.\\
Example: $\pi=(1 \,2)(3\, 4)(5\, \bar{\bar{6}})
(\bar{7} \, \bar{\bar{\bar{8}}})$
then $\pi'=(1 \,2)(3\, 4)( \bar{\bar{6}})
(\bar{7} \, \bar{\bar{\bar{8}}})(9\, 5)$.
            \item [3.1.2] If $2p+2$ is not
            a  1-circular succession of $|\pi|$ then:
\begin{itemize}
                  \item [a)] If $c(2p+1)=2$
                   then  $\pi'$ is obtained by  deleting the
                   cycle $(2p+1,\pi(2p+1))$ and inserting
                   $2p+1$ just before the letter $\widehat{2p+2}$ and creating the cycle
$( \lambda n, |\pi|(2p+1))$ where
$\lambda=\sign_{\pi}(|\pi|(2p+1))$.\\
Example: $\pi=(1 2)(3 4)(5 \bar{\bar{\bar{8}}})
(\bar{\bar{6}}\bar{7} \, )$
 then $\pi'=(1 2)(3 4)(5 \, \bar{\bar{6}}\bar{7} \, )
 ( \bar{\bar{\bar{9}}} \, 8)$.
  \item [b)] If
                   $c(2p+1)>2$. Let $a= |\pi|^{-1}(2p+1)$ and
                   $\xi = \sign_{\pi}(a)$ then $\pi'$ is obtained by
                    deleting the letter $\xi\cdot a$ and creating
                    the cycle $(\xi n, a)$.\\
Example: $\pi=(1 2)(3 4)(5 \bar{\bar{\bar{8}}}
\bar{\bar{6}}\bar{7} \, )$
then $\pi'=(1 2)(3 4)(5
\bar{\bar{\bar{8}}}\bar{\bar{6}})(\bar{9} \, 7)$.
                \end{itemize}
  \end{itemize}
        \item [3.2] If $\sign_{\pi}(2p+1)=\gamma\not=1$ then
            \begin{itemize}
             \item [3.2.1] If  $c(2p+1)=1$ then  $\pi'$ is
             obtained by  deleting the letter
             $ \gamma\cdot (2p+1)$ and creating the
             cycle $(\gamma n,2p+1)$.\\
Example: $\pi=(1 2)(3 4)(\bar{\bar{5}})
(\bar{6} \, \bar{\bar{8}} \,7)$ then $\pi'= (1 2)(3 4)
(\bar{\bar{9}}\, 5)(\bar{6} \, \bar{\bar{8}} \,7) $.
             \item [3.2.2] If $c(2p+1)\not=1$.
             Let $a= |\pi|^{-1}(2p+1)$
             and $\gamma = \sign_{\pi}(a)$
             then $\pi'$ is obtained by  deleting the letter
             $\gamma \cdot a$
             and creating the cycle $( \gamma  n, a)$.\\
Example: $\pi=(1 2)(3 4)(\bar{\bar{5}} \,
\bar{\bar{8}} \,7\, \bar{6})$
then $\pi'= (1 2)(3 4)(\bar{\bar{5}} \,
\bar{\bar{8}} \,7\, )(\bar{9}\,6)$.
            \end{itemize}
             \end{itemize}
            \end{itemize}

Here is the inverse algorithm of $\tau_\ell: \pi'\mapsto (\varepsilon, k, \pi)$.
Denote by $c(n)$ the length of the cycle of $\pi'$ containing
 $n$. In what follows we write $\rho=\sign_{\pi'}(n)$
 and $\widetilde{k}=\sign_{\pi'}(k)\cdot k$.

\begin{itemize}
  \item If $c(n)\geq 3$ or $c(n)=1$ or $c(n)=2$ and
  $\sign_{\pi'}(|\pi'|(n))\neq 1$ then
  $\varepsilon= \sign_{\pi'}(n)$ and $k=|\pi'|^{-1}(n)$ and
  we obtain  $\pi$ by deleting the letter $\widetilde{n}$.
  \item If $c(n)=2$ and $\sign_{\pi'}(|\pi'|(n))= 1$
   then $\varepsilon=1$ and $k=n$.
    Let $p$ be the smallest integer such that
     the transposition $(2p+1, \,2p+2)$
is not a cycle of $\pi'$.

\begin{itemize}
  \item [a.] If $\pi'(n)=2p+1$ and $\rho=1$ then
  we delete the cycle containing $n$ and insert
  the letter $2p+1$  just before the letter $\widetilde{2p+2}$.
  \item [b.] If $\pi'(n)\not=2p+1$ and $\sign_{\pi'}(2p+1)=1$ and $|\pi'|(2p+1)=2p+2$,
   we first delete $2p+1$ and the
    cycle containing $n$, then
  create the cycle $(\rho\cdot \pi'(n), 2p+1)$.
  \item [c.] If $|\pi'|(2p+1)\not=2p+2$ and $\pi'(n)\not=2p+1$ then we delete the
   cycle containing  $n$ and then insert the  letter
    $\rho\cdot \pi'(n) $
  before the letter $\widetilde{2p+1}$.
  \item [d.] If $\pi'(n)=2p+1$ and $\rho\not=1$ then we delete
  the cycle containing $n$ and create the  cycle containing the
  single letter $2p+1$ with the sign
  $\rho$.
\end{itemize}
 \end{itemize}

For example,  the mapping
 $\tau_2:
 \left(\mathcal{C}_{2} \times [3]\times\mathcal{D}_{2,2}\right)
 \setminus F_3\longrightarrow \mathcal{D}_{2,3} \setminus E_3$, where
  $E_3=\emptyset$ and $F_3=(1,3,(1\,2))$, is given in the following table.
\begin{center}
\begin{tabular}{|c|c|c|c|c|c|c|c|c|c|}
  \hline
  &&&&&\vspace{-10 pt}&\\
$\pi \setminus (\varepsilon, k)$& $(1,1)$ & $(1,2)$ & $(1,3)$& $(\zeta,1)$ & $(\zeta,2)$ & $(\zeta,3)$\\
\hline
&&&&&\vspace{-10 pt}&\\
(12) &  (132) & (123) &  &
  $(1\bar{3}2)$ & $(12\bar{3})$ & $(12)(\bar{3})$ \\
  \hline
&\vspace{-10 pt}&&&&&\\
$(\bar{1}2)$&$(\bar{1}32)$&$(\bar{1}23)$& $(\bar{1})(3 2)$&$(\bar{1}\bar{3}2)$&$(\bar{1}2\bar{3})$&$(\bar{1}2)(\bar{3})$\\
\hline
&&&&\vspace{-10 pt}&&\\
$(1\bar{2})$&$(13\bar{2})$& $(1\bar{2}3)$&$(1 3)(\bar{2})$&
$(1 \bar{3}\bar{2})$&$(1 \bar{2}\bar{3})$&$(1\bar{2})(\bar{3})$\\
\hline
&&\vspace{-10 pt}&&&&\\
 $(\bar{1}\bar{2})$&$(\bar{1}3\bar{2})$&
 $(\bar{1}\bar{2}3)$&$(\bar{3} 2)(\bar{1})$&$(\bar{1} \bar{3}\bar{2})$&$(\bar{1} \bar{2}\bar{3})$&$(\bar{1}\bar{2})(\bar{3})$\\
\hline
&&\vspace{-10 pt}&&&&\\
 $(\bar{1})(\bar{2})$&$(\bar{1}3)(\bar{2})$&$(\bar{1})
 (\bar{2}3)$&$(\bar{3} 1)(\bar{2})$&$(\bar{1} \bar{3})(\bar{2})$&$(\bar{1})( \bar{2}\bar{3})$&$(\bar{1})(\bar{2})(\bar{3})$\\
\hline
\end{tabular}
\end{center}
\medskip
\subsection{Combinatorial proof of \eqref{drec3}}
By Theorem~\ref{k-fixe-point}
 the coefficient $d_{\ell,n}^m$ equals the cardinality of $D_{\ell,n}^m$.
We are going to establish a bijection
$\Phi: (\rho,\alpha,\si ) \longmapsto \si'$ from
$\mathcal{C}_\ell\times [n]\times D_{\ell,n-1}^m$ to
$D_{\ell,n}^m \cup D_{\ell,n-2}^{m-1}$.

Let $\sigma=|\pi|$ and $\sigma'=|\pi'|$. The cyclic factorization of  $\pi'$
is obtained from  that of
 $\pi'$ as follows:
\begin{enumerate}
\item If $\alpha=n$, $\rho=1$ and $1 \in FIX (\si)$,
 we get  $\si' \in D_{\ell,n-2}^{m-1} $ by deleting
the cycle (1) and decreasing
all other letters by 1.
\item If $\alpha=n$  and $\rho \not = 1$ then we create the  cycle $(\rho n)$.
In this case $\si' \in D_{\ell,n}^m $ and  the cycle containing $n$
is of length 1 but  $n$ is not a fixed point of  $\si'$.
\item  If $\alpha=n$, $\rho=1$ and
 $1 \not \in FIX (\si)$, then we delete   $\pi(1)$ from its cycle
  and create a new cycle $(\gamma n, \sigma(1))$ where
  $\gamma=\sign_{\si}(\sigma(1))$. In this case $\pi'(n)>m$.
\item  If $\alpha < n$ then we insert the letter $\rho n$ just before $\alpha$.
\end{enumerate}
\noindent To show that the mapping  $\Phi$ is a bijection we
construct its inverse as follows.
\begin{enumerate}
\item If
$\si' \in D_{\ell,n-2}^{m-1} $ then $\alpha=n$, $\rho=1$ and $\si$ is
obtained from
 $\si'$ by adding  1 to all letters and
creating the cycle (1).
\item  If $\si' \in D_{\ell,n}^m $ and the cycle containing
$n$ is of length 1, then let $\alpha=n$, $\rho=\sign_{\si'}(n)$
  and $\si$ is obtained from  $\si'$  by deleting the letter $\rho n$.
\item
If $\si'\in D_{\ell,n}^m $ and the  cycle containing
 $n$ is of length  2 with  $\pi'(n) > m$
 then let $\alpha=n$, $\rho=1$  and $\si$ is obtained from
  $\si'$  by deleting the letter $n$ and inserting the letter
   $\gamma \sigma'(n)$ just after  1 where $\gamma=\sign_{\pi'}(n)$.
\item In all other cases,  let
$\alpha=\sigma'(n)$, $\rho=\sign_{\si'}(n)$ and
 $\si$ is obtained from  $\si'$ by just deleting the letter $\rho n$.
\end{enumerate}

For $n=9; m=4; \ell=3$  we give some examples to
illustrate the above bijection.
\begin{itemize}
\item $\si'=(1 \bar{\bar{5}})(2)(3 \bar{6})(4\bar{7}) \in
D_{\ell,n-2}^{m-1}$ then $\alpha=9$, $\rho=1$ and
 $\si=(1)(2 \bar{\bar{6}})(3)(4 \bar{7})(5\bar{8}) $.

\item $\si'=(1 \bar{\bar{5}})(2\bar{8})(3 \bar{6})(4\bar{7})
(\bar{\bar{9}})\in D_{\ell,n}^{m}$ then $\alpha=9$,
 $\rho=\zeta^2$ and  $\si=(1 \bar{\bar{5}})(2\bar{8})(3 \bar{6})(4\bar{7})$.

 \item $\si'=(1 5)(2\bar{8})(3 )(4\bar{7})(\bar{\bar{9}}6)\in D_{\ell,n}^{m}$
  then $\alpha=9$, $\rho=1$ and  $\si=(1 \bar{\bar{6}}5)(2\bar{8})(3 )
  (4\bar{7})$.

\item $\si'=(1 \bar{\bar{5}})(2\bar{8})(\bar{6} )
(4)(\bar{\bar{9}}\bar{7}3)\in D_{\ell,n}^{m}$ then
 $\alpha=7$, $\rho=\zeta^2$ and  $\si=(1 \bar{\bar{5}})
 (2\bar{8})(\bar{6} )(4)(\bar{7}3)$.
\end{itemize}

\vspace{40pt}
{\bf Acknowledgments}:
This work was done during the visit of the first author
to Institut Camille Jordan (Universit\'{e} Lyon 1) in the fall of 2007.  This visit  was
supported by  a  scholarship of  AUF (Agence universitaire de la francophonie).
Both authors would like to thank the referee for his/her careful reading of a previous
version of this paper.

\end{document}